\documentclass{article}

\usepackage{a4wide}
\usepackage{enumerate}

\usepackage{graphicx, float}
\usepackage{subfigure}

\usepackage{amssymb}

\usepackage{amsthm}

\usepackage{amsmath}
\usepackage{enumerate}

\usepackage{boolexpr, xstring}

\usepackage[table]{xcolor}

\newtheorem{theorem}{Theorem}[section]
\newtheorem{definition}[theorem]{Definition}
\newtheorem{proposition}[theorem]{Proposition}
\newtheorem{lemma}[theorem]{Lemma}
\newtheorem{claim}[theorem]{Claim}
\newtheorem{corollary}[theorem]{Corollary}
\newtheorem{conjecture}[theorem]{Conjecture}

\newtheorem{observation}[theorem]{Observation}

\title{Bounded diameter monochromatic component covers}
\author{Alexey Pokrovskiy\thanks{	Department of Mathematics, 
		University College London, 
		Gower Street, London WC1E~6BT, UK. 
		Email: \texttt{dralexeypokrovskiy@gmail.com}}}

\begin{document}
\maketitle
\begin{abstract}
Ryser conjectured that every $r$-edge-coloured complete graph can be covered by $r-1$ monochromatic components. Motivated by a question of Austin in analysis, Mili\'cevi\'c  predicted something stronger --- that every $r$-edge-coloured complete graph can be covered by $r-1$ monochromatic components \emph{of bounded diameter}. Here we show that the two conjectures are equivalent. As immediate corollaries we obtain new results about Mili\'cevi\'c's Conjecture, most notably that it is true for $r=5$. We also obtain several new cases of a generalization of Mili\'cevi\'c's Conjecture to non-complete graphs due to DeBiasio-Kamel-McCourt-Sheats.

\end{abstract}
\section{Introduction}
This paper is about a conjecture commonly attributed to Ryser (it is unclear if Ryser actually made this conjecture~\cite{best2018did}. To our knowledge it first appeared in the thesis of Henderson \cite{hend}).  Recall that for a graph $G$, the independence number of $G$ (denoted by $\alpha(G)$) denotes the size of a maximum set of vertices containing no edges.
\begin{conjecture}[Ryser]\label{Conjecture_Ryser}
Every $r$-edge-coloured graph can be covered by $(r-1)\alpha(G)$ monochromatic components.
\end{conjecture}
This conjecture is much more commonly phrased in terms vertex covers of hypergraphs. It is equivalent to the statement that ``every $r$-uniform, $r$-partite hypergraph $\mathcal H$ has  $\tau(\mathcal H)\leq (r-1)\nu(\mathcal H)$''. See Section~\ref{Section_Hypergraph_Equivalence} for an explanation of this equivalence.  
 In this paper we will talk about the ``coloured-graph'' version of Ryser's Conjecture because our aim will be to find a generalization of the conjecture in that setting.

Results about Ryser's Conjecture are limited. For $r=2$ it is equivalent to the famous K\"{o}nig-Hall Theorem. For $r=3$ it was proved by Aharoni using a very surprising proof relying on methods from topology \cite{ahro}. For $r\geq 4$ it is open in general. However, in the case $\alpha(G)=1$ (i.e. when $G$ is a complete graph), the conjecture has been proved for $r\leq 5$ (see \cite{duchet1979representations, gyarfas1977partition, tuz_un, tuz2}).  

The goal of this paper is to study a strengthening of Ryser's Conjecture posed by Mili\'cevi\'c. He conjectured that in the complete graph case ($\alpha(G)=1$), the diameters of the components in Conjecture~\ref{Conjecture_Ryser} can be bounded by some constant depending on $r$.
\begin{conjecture}[Mili\'cevi\'c, \cite{milicevic2019covering}]\label{Conjecture_Milicevic}
For all $r$ there  a constant $D(r)$ such that every $r$-edge-coloured complete graph can be covered by $r-1$ monochromatic components of diameter $\leq D(r)$.
\end{conjecture}
We remark that in this conjecture  and its relatives, one can replace ``monochromatic components'' by ``monochromatic trees'' without changing the nature of the conjecture (using the fact that any connected graph of diameter $\le d$ contains a spanning tree of diameter $\le 2d$).  For $r=2$, it is well known that for every graph $G$, either $G$ or the complement of $G$ is connected of diameter $\leq 3$ i.e. $D(2)\leq 3.$
Mili\'cevi\'c proved this conjecture for $r\leq 4$ by showing that $D(3)\leq 8, D(4)\leq 80$ \cite{milicevic2015commuting, milicevic2019covering}. 
DeBiasio-Kamel-McCourt-Sheats improved the bounds, by showing that $D(3)\leq 4, D(4)\leq 6$.

The motivation for Mili\'cevi\'c's Conjecture comes from analysis. 
For a function $f$ on a metric space $(d,X)$ and $\lambda\in (0,1)$, we say that $f$ is a $\lambda$-contraction if $d(f(x), f(y))\leq \lambda d(x,y)$ for all $x,y \in X$. 
Austin posed the following problem.
\begin{conjecture}[Austin, \cite{austin2005contractive}]
For all $\lambda\in (0,1)$, suppose that $\{f_1, \dots, f_r\}$ is a  family of $\lambda$-contractions on a complete metric space which pairwise commute with each other. Then  $\{f_1, \dots, f_r\}$ 
have a common fixed point.
\end{conjecture}
Austin proved this conjecture for $r=2$ \cite{austin2005contractive}.
Mili\'cevi\'c's proved the conjecture for $r=3$ and for sufficiently small $\lambda$. A key ingredient in this proof was the $r=3$ case of Conjecture~\ref{Conjecture_Milicevic}, which motivated Mili\'cevi\'c to make the general conjecture.

Ryser's Conjecture is about non-complete graphs, so one may naturally wonder whether Conjecture~\ref{Conjecture_Milicevic} should extend to non-complete graphs too. Such an extension was conjectured by DeBiasio, Kamel, McCourt, Sheats.
\begin{conjecture}[DeBiasio-Kamel-McCourt-Sheats, \cite{debiasio2020generalizations}]\label{Conjecture_Non_complete}
For all $\alpha \geq 1$, there exists $d = d(\alpha)$ such that for all $r \geq 2$, if $G$ is a graph with $\alpha(G) = \alpha$, then in every $r$-coloring of $G$, there exists $(r - 1)\alpha$ monochromatic components of diameter $\leq d$ covering $G$. 
\end{conjecture}
This conjecture is known  for $(\alpha, r)=(1,2), (1,3), (1,4)$ (where it is just Conjecture~\ref{Conjecture_Milicevic}), $(\alpha, r)=(2, 2)$, where it was proved by DeBiasio, Kamel, McCourt, Sheats~\cite{debiasio2020generalizations}, and $r=2$, $\alpha\in \mathbb{N}$ where it was proved by DeBiasio, Girão,  Haxell,  Stein~\cite{debiasio2025bounded}. Additionally there has been work~\cite{english2021low, gyarfas2025bounded, gyarfas20252} on figuring out how small the diameter $d$ can be for various $r$ and $\alpha$.

The goal of this paper is to show that for every fixed $\alpha, r$ Ryser's Conjecture is equivalent to Conjecture~\ref{Conjecture_Non_complete}.
\begin{theorem}\label{Theorem_main}
For fixed $r, \alpha$ the following are equivalent.
\begin{enumerate}[(i)]
\item Ryser's Conjecture holds for all $r$-edge-coloured graphs with $\alpha(G)=\alpha$.
\item Every $r$-edge-coloured graph with $\alpha (G)=\alpha$ can be covered by $(r-1)\alpha(G)$ monochromatic components of diameter $\leq  9^{r^{2^{(r+\alpha)^{4r}}}}$.
\end{enumerate}
\end{theorem}

The question of whether such a reduction might be possible was actually raised in~\cite{debiasio2025bounded}.  
This theorem largely reduces Conjectures~\ref{Conjecture_Milicevic} and~\ref{Conjecture_Non_complete}  to the much better-studied conjecture of Ryser. As a consequence, all known results about Ryser's Conjecture carry over to the conjectures of Mili\'cevi\'c and DeBiasio-Kamel-McCourt-Sheats.
\begin{corollary}\label{Corollary1}
Mili\'cevi\'c's Conjecture holds for $r=5$.
\end{corollary}
This follows from Theorem~\ref{Theorem_main} combined with Tuza's proof of Ryser's Conjecture for $r=5$, $G=K_n$.

By combining Theorem~\ref{Theorem_main} with Aharoni's topological proof of Ryser's Conjecture for $r=3$, we   obtain Conjecture~\ref{Conjecture_Non_complete} for $\alpha =3$:
\begin{corollary}\label{Corollary3}
Every $3$-edge-coloured graph $G$ can be covered by  $2\alpha(G)$ monochromatic components of diameter $\leq  9^{3^{2^{\alpha(G)^{12}}}}$.
\end{corollary}
 
We remark that our proof methods naturally prove something a little stronger than Theorem~\ref{Theorem_main}. Instead of coloured graphs, our proof works with metric spaces and so our main technical result  (Lemma~\ref{Lemma_Stability_Transferrence}) is a statement about families of metrics defined on the same finite set.

\section{Proof ideas and examples}
In this section we illustrate the basic idea of the proof of this paper. We focus just on how our methods work in the (previously known) case when $r=3$, $\alpha=1$.

The method we use to obtain the reduction in Theorem~\ref{Theorem_main}  can be summarized as ``we show that if Ryser's Conjecture is true, then a certain case-analytic proof approach to it can always be executed whilst controlling the diameters of the monochromatic components''. The fact that for every fixed $r, \alpha$, Ryser's Conjecture can be proved or disproved using a finite case analysis is a known (folklore) result. 
Below is an example of what we mean by a ``case analytic'' proof of the conjecture for $r=3$.

\begin{figure}[h!]
 \centering
    \includegraphics[width=0.9\textwidth]{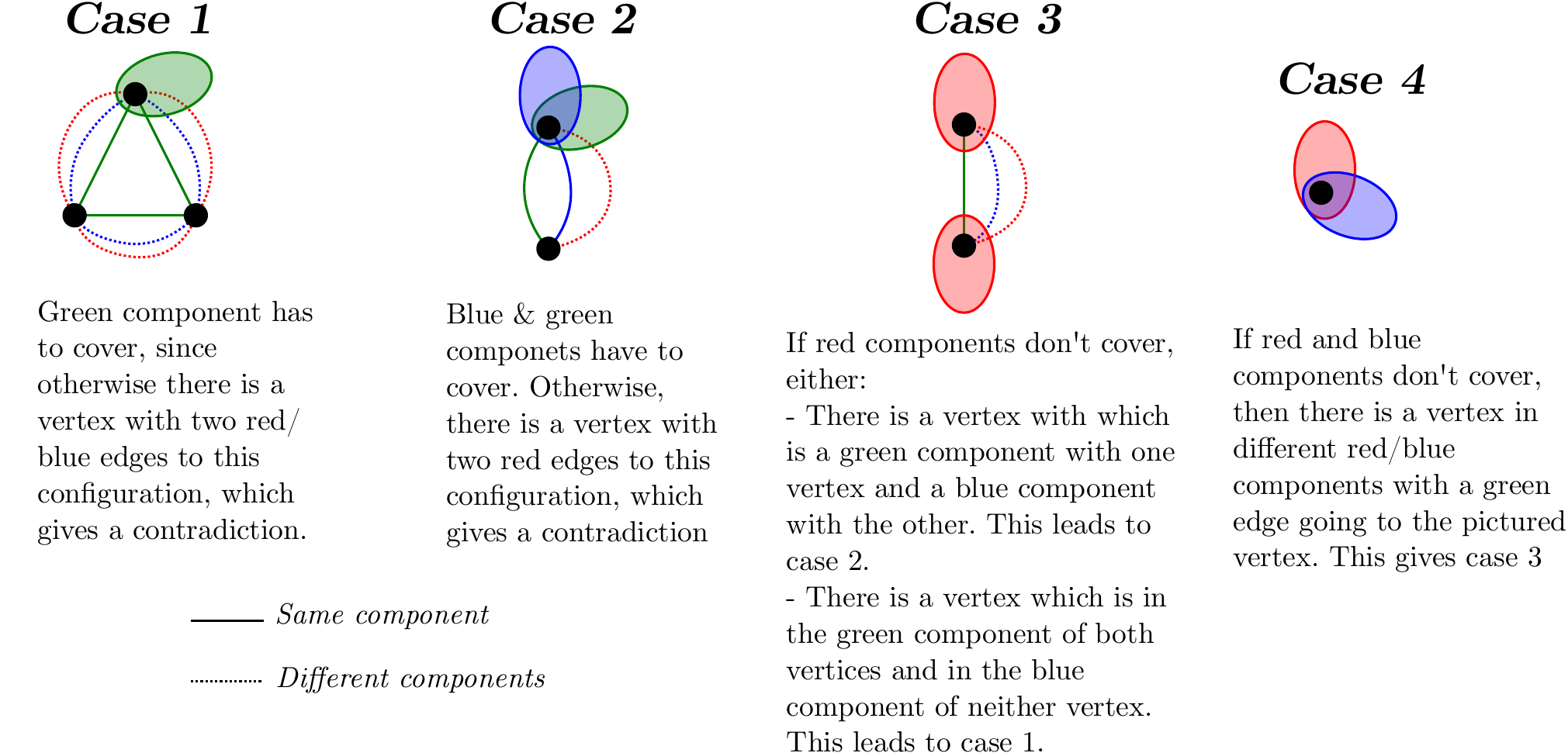}
  \caption{\footnotesize A proof of of Ryser's Conjecture for $r=3$.  Each case represents a configuration of monochromatic components that can occur in a $3$-coloured $K_n$. Each case has $\leq 2$ special monochromatic components which should cover the $3$-coloured $K_n$ (the shaded coloured ovals). }
\label{FigureColouredGraph}
\end{figure}

The above proof doesn't control the diameter of the monochromatic components it produces. However, by suitably modifying it, it is possible to obtain such a control. 

\begin{figure}[H]
 \centering
    \includegraphics[width=0.9\textwidth]{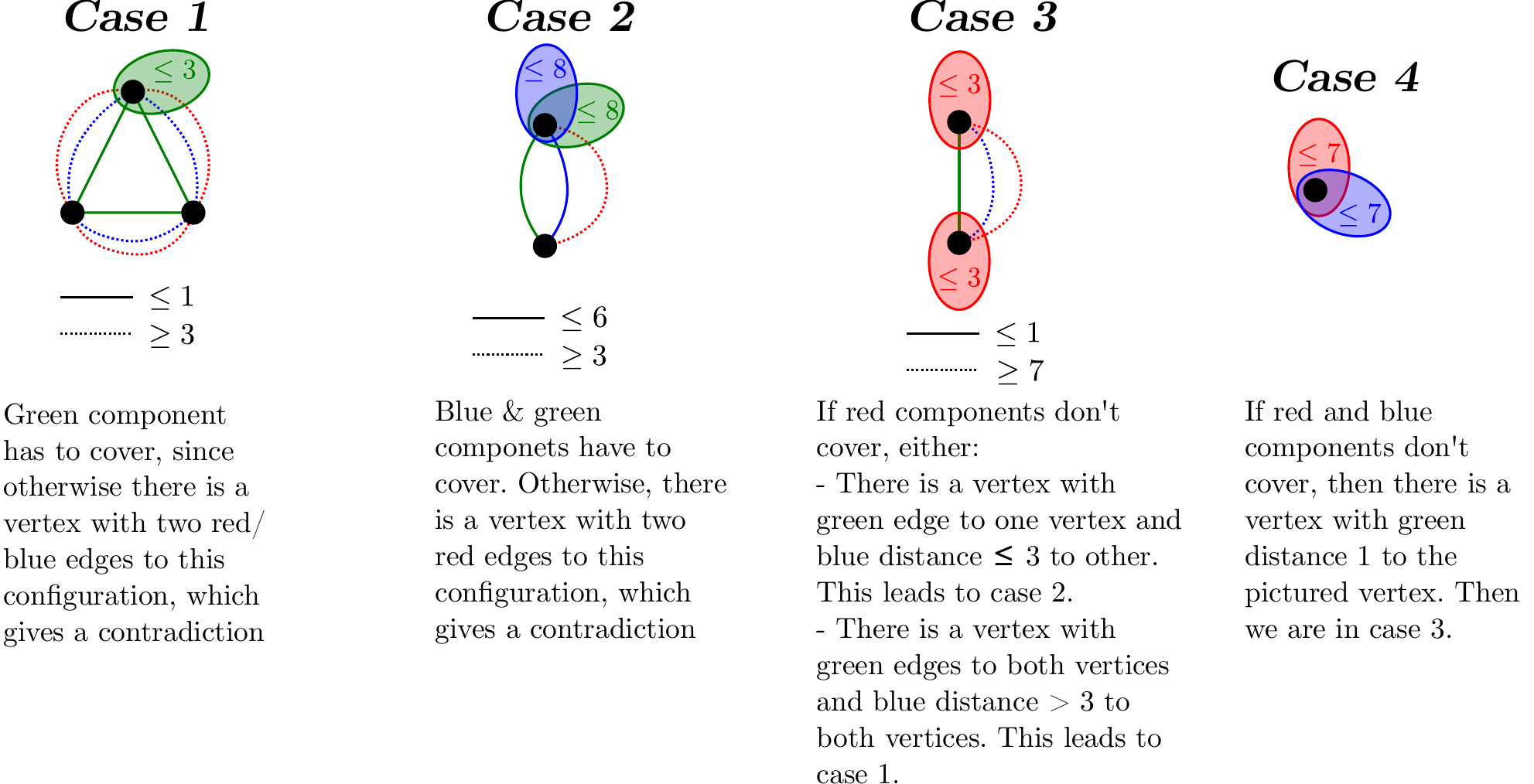}
  \caption{\footnotesize A proof of Conjecture~\ref{Conjecture_Milicevic} for $r=3$. Each case represents a configuration of vertices in a $3$-coloured $K_n$ with restrictions on the coloured distances between them. Each case has $\leq 2$ special monochromatic components of bounded diameter which should cover the $3$-coloured $K_n$ (the shaded coloured ovals).}
\label{FigureMetric}
\end{figure}

Notice that there is no difference between the pictures that are drawn --- the only differences between the proofs are in the arguments that go with the pictures. 
The basic idea of the proof of Theorem~\ref{Theorem_main} is formally show that (for fixed $\alpha, r$) a proof of Ryser's Conjecture can be modified in a way to control the diameter of the monochromatic components it produces.
There are a number of steps we go through to achieve this:
\begin{itemize}
\item In Section~\ref{Section_Hypergraph_Equivalence}, we recall that Ryser's Conjecture has an equivalent statement in terms of matchings and covers of $r$-partite hypergraphs. 
\item In Section~\ref{Section_ryser_Stability}, we recall the definition of a ``Ryser stable sequence''. This is a definition introduced in~\cite{abu2016matchings} which formalizes the idea of ``a proof of Ryser's Conjecture using a case analysis''. We also prove Lemma~\ref{Lemma_RyserStabilityEquivalence} which shows that a Ryser stable sequence exists $\iff$ Ryser's Conjecture is true for particular $\alpha,r$.
\item In Section~\ref{Section_metrics}, we introduce some notation about finite metric spaces.  
\item Section~\ref{Section_approximate_duals} is where the main proof of the paper takes place. First a key new definition of an ``$(a,b)$-dual'' of a hypergraph is introduced. Then, we prove the main technical lemma of the paper (Lemma~\ref{Lemma_Stability_Transferrence}). Then we use this to deduce two lemmas which show that if there exists a Ryser-stable sequence then one can control diameters of components in Ryser's conjecture. Theorem~\ref{Theorem_main} follows easily from these.  
\end{itemize}

\section{Coloured graphs and $r$-partite hypergraphs}\label{Section_Hypergraph_Equivalence}
There is a a duality between $r$-uniform, $r$-partite hypergraphs, and edge-colourings of graphs. This duality was first observed in Erd\H{o}s, Gy\'arf\'as, and Pyber's paper \cite{erdHos1991vertex}. The duality works as follows:

\textbf{Hypergraphs $\to$ coloured graphs:} For a $r$-partite, $r$-uniform hypergraph, define a  $r$-edge-coloured multigraph $G(\mathcal H)$ by letting $V(G(\mathcal H))=E(\mathcal H)$, with a colour $i$ edge between $e,f\in V(G(\mathcal H))$ whenever $e$ and $f$ intersect in partition $i$.

\textbf{Coloured graphs $\to$ hypergraphs:} For a $r$-edge-coloured multigraph $G$, define an $r$-partite hypergraph $\mathcal H(G)$ as follows. Let  $V(\mathcal H)$ be the set of monochromatic components of $G$, with partition $i$ of $\mathcal H(G)$ consisting of the colour $i$ components.
For each vertex $v\in V(G)$, we place an edge in $\mathcal H(G)$ whose vertex in part $i$ is the colour $i$ component containing $v$.
Under these transformations, various notions carry over as follows:
\begin{center}
\begin{tabular}{| l| l |}
 \hline
  Hypergraph world & Edge-coloured graph world \\
  \hline
  matching & independent set  \\
  edge & vertex  \\
  vertex & monochromatic component  \\
  partition & colour  \\
  sub-hypergraph & induced subgraph  \\
  vertex-cover & monochromatic component cover  \\
   \hline
\end{tabular}
\end{center} 
Using these we immediately obtain the equivalence between the two forms of Ryser's Conjecture:
Using this correspondence we immediately get the equivalence between the two forms of Ryser's Conjecture. 
\begin{proposition}\label{Proposition_HypergraphEquivalence}
For fixed $\nu, r\in \mathbb N$, the following are equivalent:
\begin{itemize}
\item Every $r$-uniform, $r$-partite hypergraph $\mathcal H$ with matching number $\nu$ has $\tau(\mathcal H)\leq (r-1)\nu$.
\item Every $r$-edge-coloured multigraph $G$ with independence number $\nu$ can be covered by $(r-1)\nu$ disjoint monochromatic components.
\end{itemize}
\end{proposition}

\section{Ryser stability}\label{Section_ryser_Stability}
Here we formalize what we mean by ``case analytic proof of Ryser's Conjecture''. The key notions of \emph{Ryser-stable hypergraphs} and of \emph{Ryser-stable sequences of hypergraphs}. These were first introduced in the PhD thesis of Abu-Kazneh~\cite{abu2016matchings}, and everything in this section reproves and expands on lemmas from there.
Informally, \emph{Ryser-stable hypergraph} means ``a single step of a case-analytic proof of Ryser's Conjecture'', whereas  \emph{Ryser-stable sequences of hypergraphs} means ``a case-analytic proof of Ryser's Conjecture''. The full definitions are as follows:
\begin{definition}
Let $\mathcal H, \mathcal H_1, \dots, \mathcal H_{\ell}$ be  $r$-uniform, $r$-partite hypergraphs.
We say that $\mathcal H$ is  $c$-Ryser-stable relative to $\mathcal H_1, \dots, \mathcal H_{\ell}$ if there is a vertex cover $C$ of $\mathcal H$ of size $c$ such that for any $r$-uniform, $r$-partite hypergraph $\mathcal H'$ containing $\mathcal H$, at least one of the following holds
\begin{itemize}
\item $\mathcal H'$ is covered by $C$.
\item $\mathcal H'$ contains a copy of $\mathcal H_i$ for some $i$.
\end{itemize} 
\end{definition}
If $C$ is a cover as in the above definition, we say that $C$ \emph{witnesses} the $c$-Ryser stability of $\mathcal H$ relative to $\mathcal H_1, \dots, \mathcal H_{\ell}$. Knowing that a hypergraph $\mathcal H$ is $c$-Ryser-stable relative to $\mathcal H_1, \dots, \mathcal H_{\ell}$ shows that all hypergraphs containing $\mathcal H$ either satisfy Ryser's Conjecture or that they contain one of the other hypergraphs $\mathcal H_1, \dots, \mathcal H_{\ell}$.  This can be used a single step in a case-analytic proof of the conjecture by chaining together several Ryser-stable hypergraphs.
\begin{definition}
Let $\mathcal H_1, \dots, \mathcal H_n, \mathcal R_1, \dots, \mathcal R_k$ be   $r$-uniform, $r$-partite hypergraphs.
We say that $\mathcal H_1, \dots, \mathcal H_n$ is a $c$-Ryser-stable sequence relative to  $\mathcal R_1, \dots, \mathcal R_k$ if for each $i=1, \dots, n$, the hypergraph $\mathcal H_{i}$ is $c$-Ryser-stable relative to $\mathcal H_1, \dots, \mathcal H_{i-1}, \mathcal R_1, \dots, \mathcal R_k$.
\end{definition}

\begin{figure}[h]
 \centering
    \includegraphics[width=0.9\textwidth]{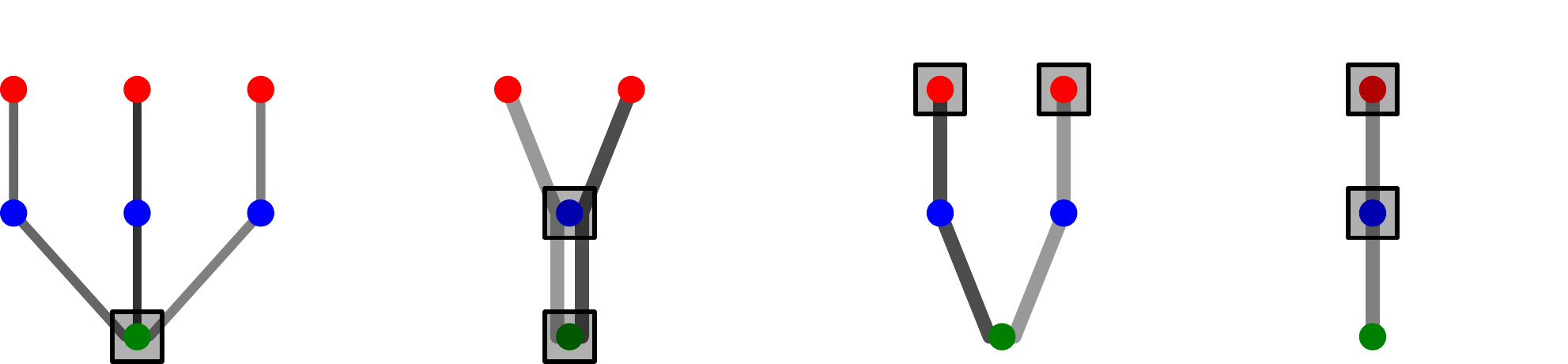}
  \caption{\footnotesize An example of a sequence of 3-uniform, 3-partite hypergraphs that is $2$-Ryser-stable sequence relative to $\mathcal{M}_{3, 2}$ (the 3-uniform matching of 2 edges). The   covers witnessing 2-Ryser stability are given by square vertices in each hypergraph.}
\label{FigureColouredGraph}
\end{figure}

Using the Sunflower Lemma, we can prove that Ryser's Conjecture is equivalent to the existence of Ryser-stable sequences. 
 Let $\mathcal S_{r}(s,t)$ denote the $r$-uniform sunflower with core of size $s$ and with $t$ petals, which is defined as follows: the vertex set   $V(\mathcal S_{r}(s,t))=C\cup P_1\cup \dots\cup P_t$ where $|C|=s$, $|P_1|=\dots=|P_t|=r-s$, and these sets are all disjoint. The edges of $\mathcal S_{r}(s,t)$ are the $t$ sets $C\cup P_1, \dots, C\cup P_t$. The set $C$ is called the core of the sunflower and the sets $P_i$ the petals. The following fundamental theorem shows that all hypergraphs with enough edges contain sunflowers.
 \begin{theorem}[Erd\H{o}s-Rado Sunflower Lemma]
 Every $r$-uniform hypergraph with $\geq r!t^{r+1}$ edges contains a copy of $\mathcal S_{r}(s,t)$ for some $s$.
 \end{theorem}
 We'll also need the following simple lemma for finding matchings in hypergraphs with large sunflowers.
 \begin{lemma}\label{Lemma_matching_out_of_large_sunflowers}
Let $S_1, \dots, S_a$ be core disjoint sunflowers with $(\nu+1)r$ petals. Let $M$ be a size $\nu-a+1$ matching disjoint from the cores of $S_1, \dots, S_a$. Then $\nu(S_1\cup\dots\cup S_a\cup M)\geq \nu+1$.
\end{lemma}
\begin{proof}
 Let $A_i$ be the core of $S_i$ and $B_i=\bigcup_{j\neq i} A_j$ be the union of the cores outside $A_i$, noting that $A_i\cap B_i=\emptyset$.
 Pick edges $e_1\in S_1, \dots, e_a\in S_a$ one by one with $(e_i\setminus A_i)\cap (V(M+e_1+\dots+e_{i-1})\cup B_i)=\emptyset$. 
 This is possible, since there $(\nu+1)r$ choices for each $e_i$, for all these choices $e_i\setminus A_i$  are vertex-disjoint, and $|V(M+e_1+\dots+e_{i-1})\cup B_i|\leq \nu r<(\nu+1)r$. 
 
 Now $M'=M+e_1+\dots+e_{a}$ is a size $\nu+1$ matching. Indeed, since $M$ is a matching, the only way there could be an intersection is if $e_i\cap f \neq \emptyset$ for some $f\in M+e_1+\dots+e_{i-1}$. By ``$(e_i\setminus A_i)\cap (V(M+e_1+\dots+e_{i-1})\cup B_i)=\emptyset$'' we have that $e_i\cap f\subseteq A_i$. Since $M$ is disjoint from the cores of $S_1, \dots, S_a$, we have that $f=e_j$ for some $j$. But we have $A_j\cap A_i=\emptyset$ (since $S_j, S_i$ are core-disjoint) and $(e_j\setminus A_j)\cap A_i=\emptyset$ (by choice of $e_j$). Plugging these into $e_j=A_j\cup (e_j\setminus A_j)$ gives $e_j\cap A_i=\emptyset$, which is a contradiction.
\end{proof}
 
 Use $\mathcal{M}_{r, \nu}$ to denote an $r$-uniform matching of size $\nu$ (so $\mathcal{M}_{r, \nu}=\mathcal S_{r}(0,\nu)$). Use $\mathcal E_r$ to denote the single-edge $r$-uniform hypergraph.
The following is the main result of this section. It is an extension of Corollary 2.7 from~\cite{abu2016matchings}.
\begin{lemma}\label{Lemma_RyserStabilityEquivalence}
For fixed $\nu, r\in \mathbb N$, the following are equivalent:
\begin{enumerate}[(i)]
\item Every $r$-uniform, $r$-partite hypergraph $\mathcal H$ with matching number $\nu$ has $\tau(\mathcal H)\leq (r-1)\nu$.
\item For some $\ell\leq 2^{(r+\nu)^{2r}}$, there exists a $(r-1)\nu$-Ryser-stable sequence $\mathcal H_1, \dots, \mathcal H_{\ell}$ relative to $\mathcal{M}_{r, \nu+1}$ such that the last hypergraph $\mathcal H_{\ell}$ in the sequence is just a single edge $\mathcal E_r$. 
\end{enumerate}
\end{lemma}
\begin{proof}
(i) $\implies$ (ii): 
Let $R(a,b)$ be the set of hypergraphs $R$ with $\nu(R)\leq \nu-a$, $e(R)=b$, and containing no sunflower with $(\nu+1)r$ petals and non-empty core. By the Sunflower Lemma, we have that $R(a,b)$ is empty for $b\geq r!((\nu+1)r)^{r+1}$.
Say that a hypergraph is $(a,b)$-\emph{basic} if it has no isolated vertices and consists of $a$ core-disjoint sunflowers $S_1, \dots, S_a$ with nonempty cores and $(\nu+1)r$ petals each, and a hypergraph $R\in R(a,b)$ that is disjoint from the cores of $S_1, \dots, S_a$.
\begin{claim}
$(a,b)$-\emph{basic} hypergraphs $\mathcal H$ have  $\nu(\mathcal H)\leq \nu$.
\end{claim}
\begin{proof}
A matching can use at most one edge from each $S_1, \dots, S_a$, since sunflowers with nonempty cores are intersecting. A matching can use at most $\nu(R)\leq \nu-a$ edges of $R$. So it total, it can have at most $a+(\nu-a)\leq \nu$ edges.
\end{proof}

Let $F(a,b)$ be the set of all $(a,b)$-basic hypergraphs listed in an arbitrary order.  
Let $b_0=r!((\nu+1)r)^{r+1}$. Recall the lexicographic ordering on pairs is defined by $(a,b)\succ (a',b')$ if $a>a'$ or if $a=a'$ and $b>b'$.
Consider the sequence formed by concatenating the sequences  $F(a,b)$ for $(a,b)\in [\nu]\times [b_0]$ according to ``$\succ$'' i.e. the sequence

\begin{align*}
F(\nu, b_0), \  F(\nu, b_0-1), \dots, F(\nu, 0), \\
F(\nu-1, b_0), F(\nu-1, b_0-1), \dots, F(\nu-1, 0), \\
\vdots \hspace{1cm}\\
F(0, b_0),\ F(0, b_0-1), \dots, F(0, 1).
\end{align*} 
The length of the sequence is clearly less than the number of hypergraphs with $E(\mathcal H)\leq b_0+\nu(\nu+1)r$  and $V(\mathcal H)\leq (b_0+\nu(\nu+1)r)r$ which is $\leq \binom{(b_0+\nu(\nu+1)r)r}r^{b_0+\nu(\nu+1)r}\le ((b_0+\nu(\nu+1)r)r)^{(b_0+\nu(\nu+1)r)r}=   2^{(b_0+\nu(\nu+1)r)r\log((b_0+\nu(\nu+1)r)r)}\leq 2^{(r+\nu)^{2r}}$.
The sequence  ends with the single-edge hypergraph $\mathcal E_r$ , since that is the only hypergraph in $F(0,1)$.
It remains to prove that this is a $(r-1)\nu$-Ryser-stable sequence relative to  $\mathcal{M}_{r, \nu+1}$. This is immediate from   the following claim.
\begin{claim}
Let  $\mathcal H\in F(a,b)$. Then $\mathcal H$ is $(r-1)\nu$-Ryser-stable  relative to $\mathcal{M}_{r, \nu+1}$ together with all the hypergraphs in $\bigcup_{(a',b')\succ (a,b) \text{ with } (a',b')\in [\nu]\times [b_0]}F(a',b')$.
\end{claim}
\begin{proof}
By definition of $\mathcal H\in F(a,b)$, we have $\mathcal H=S_1\cup \dots\cup S_a \cup R$ where $S_i$ are sunflowers (with nonempty core and $\geq (\nu+1)r$ petals) and $R\in R(a,b)$ (which is disjoint from the cores of $S_1, \dots, S_a$).

Define $C_1\subseteq V(\mathcal H)$ to consist of the union of the  cores of $S_1\cup \dots\cup S_a$ and define $C_2\subseteq V(\mathcal H)$ to be  a minimal cover of $R$. Let $C=C_1\cup C_2.$ We'll show that $C$ is a cover that witnesses the Ryser-stability of $\mathcal H$ (which proves the claim).

It is immediate that $C$ is a cover (since the sunflowers $S_i$ have nonempty cores). 
Since $R\in R(a,b)$, we have $\nu(R)\leq \nu-a$.
By (i) we have $|C_2|\leq (r-1)\nu(R)\leq (r-1)(\nu-a)$. Since a sunflower with $\geq 2$ petals has core at most $r-1$, we have $|C_1|\leq  (r-1)a$. These give $|C|=|C_1|+|C_2|\leq (r-1)a+ (r-1)(\nu-a)= (r-1)\nu$.

Let $\mathcal H'$ be a hypergraph containing $\mathcal H$ which isn't covered by $C$. Since $\mathcal H'$ is not covered by $C\supseteq C_1$, it contains some edge $e\not\in \mathcal H$ which is disjoint from the cores of $S_1, \dots, S_a$. Because of this, one of the following is true.
\begin{enumerate}[(1)]
\item $R+e$ contains a  sunflower with $(\nu+1)r$ petals and non-empty core. Then   $\mathcal H+e$ contains $a+1$ core-disjoint sunflowers with $(\nu+1)r$ petals. If $a+1\leq \nu$, then this would show that $\mathcal H+e$ contains a member of $F(a+1,0)$. If $a+1\geq \nu+1$, Lemma~\ref{Lemma_matching_out_of_large_sunflowers} applied with $M=\emptyset$ gives a size $\nu+1$ matching $\mathcal M_{r,\nu+1}$.
\item $\nu(R+e)=\nu-a+1$. Let $M$ be a matching in $R+e$ of this size, noting that it is core-disjoint from $S_1, \dots, S_a$. Lemma~\ref{Lemma_matching_out_of_large_sunflowers} gives a size $\nu+1$ matching $\mathcal M_{r,\nu+1}$.
\item $R+e\in R(a, b+1)$. Then we have $b+1\leq b_0$ (since otherwise $R+e$ would have  a size $(\nu+1)r$ sunflower by the Sunflower Lemma), and so $\mathcal H+e\in F(a, b+1)$.  
\end{enumerate}
Thus we have shown that in any of the cases, $\mathcal H'$ contains either $\mathcal M_{r, \nu+1}$ or a member of $F(a',b')$ with $(a',b')\succ(a,b)$ and $(a',b')\in[\nu]\times[b_0]$ as required.
\end{proof}

(ii) $\implies$ (i): Let  $\mathcal H_1, \dots, \mathcal H_{\ell}$  be a $(r-1)\nu$-Ryser-stable sequence relative to $\mathcal{M}_{r, \nu+1}$ with $\mathcal H_{\ell}=\mathcal E_r$.
Let $\mathcal H$ be a hypergraph with $\nu(\mathcal H)\leq \nu$. If $\mathcal H$ has no edges, then $\tau(\mathcal H)=0\leq (r-1)\nu$, so assume that it has edges.
Let $i$ be minimal for which there is an isomorphic copy $\mathcal H_i\subseteq \mathcal H$ (it exists because $\mathcal H_{\ell}=\mathcal E_r$ which every hypergraph contains). There is some cover $C$ witnessing $(r-1)\nu$-Ryser-stability of $\mathcal H_i$ relative to $\mathcal H_1, \dots, \mathcal H_{i-1}, \mathcal{M}_{r, \nu+1}$. In particular $|C|\leq (r-1)\nu$. If $C$ covers $\mathcal H$, then we are done. Otherwise, by the definition of Ryser-Stability, $\mathcal H$ contains a copy of one of $\mathcal H_1, \dots, \mathcal H_{i-1}, \mathcal{M}_{r, \nu+1}$. It cannot contain a copy of $\mathcal H_1, \dots, \mathcal H_{i-1}$ by minimality of $i$. So it contains $\mathcal{M}_{r, \nu+1}$, contradicting $\nu(\mathcal H)\leq \nu$.

\end{proof}

\section{Metric spaces}\label{Section_metrics}
In this paper we deal with finite metric spaces i.e.  pairs $(V, d)$ where $V$ is a finite set and $d:V\times V\to \mathbb{R}$ is a function satisfying the axioms of a metric. 
The set $V$ will be referred to as  the \emph{vertex set of the metric}. 
Two metric spaces $(V,d)$ and $(V',d')$ are isomorphic if there is a bijection $\theta:V\to V'$ between their vertex sets for which $d=d'\circ \theta$.
Given two metrics $(V,d)$ and $(V',d')$ we say that $(V,d)$ contains an \emph{isomorphic copy} of $(V',d')$ if there is some subset $S\subseteq V$ with $(S,d)$ isomorphic to $(V',d')$

We say that $(V, d_1, \dots, d_r)$ is a \emph{family of metrics with common vertex set} if  $(V, d_1), \dots, (V, d_r)$ are metrics  on the set $V$.

Recall that for a connected graph $G$, we can define the \emph{graph metric} $d_G: V(G)\times V(G)\to \mathbb{R}$ where $d_G(x,y)$ is the length of the shortest $x$ to $y$ path in $G$. We'll need a version of   the graph metric for disconnected graphs as well. Perhaps the most natural definition is to set ``$d_G(x,y)=\infty$'' when $x$ and $y$ are in different components. However, to avoid infinite quantities, we'll instead use the following finite version:
\begin{equation}\label{Eq_GraphMetric}
d_G(x,y)=\begin{cases}\text{length of shortest $x$ to $y$ path if it exists}\\ \text{$|V(G)|$ otherwise}\end{cases}
\end{equation}
Note that for connected graphs $d_G$ is just the usual graph metric. It is also a metric in general.
\begin{proposition}\label{Proposition_Metric}
For any graph $G$, $d_G:V(G)\times V(G)\to \mathbb{R}$ is a metric
\end{proposition}
\begin{proof}
We need to show that $d(x,y)\leq d(x,z)+d(z,y)$ for all $x,y,z$.
If $x,y,z$ are in the same component then this holds due to the usual graph metric being a metric. 
If $x,y$ are in the same component and $z$ is in a different component then we have $d(x,y)\leq 2|V(G)|= d(x,z)+d(z,y)$.
 If $x$ and $z$ are in the same component and $y$ in a different one then $d(x,y)= |V(G)|\leq d(x,z)+|V(G)|= d(x,z)+d(z,y)$.
 If $y$ and $z$ are in the same component and $x$ in a different one then $d(x,y)= |V(G)|\leq d(y,z)+|V(G)|= d(x,z)+d(z,y)$.
 If $x,y$ and $z$ are all different components then $d(x,y)= |V(G)|\leq 2|V(G)|= d(x,z)+d(z,y)$.
\end{proof}
Now given any $r$-edge-coloured graph $G$, we can think of it as a family of $r$ metrics with vertex set $V(G)$. Indeed, letting $G_i$ be the subgraph of $G$ consisting of colour $i$ edges, we have that $(V(G), d_{G_1}, \dots, d_{G_r})$ is this family.

\section{Approximate duals of hypergraphs}\label{Section_approximate_duals}
In Section~\ref{Section_Hypergraph_Equivalence} we gave a duality between $r$-uniform, $r$-partite hypergraphs and $r$-coloured graphs. This duality can show that the hypergraph version of Ryser's Conjecture is equivalent to the coloured graph version. But it isn't strong enough to say anything about Conjecture~\ref{Conjecture_Milicevic} because  the equivalence doesn't notice the diameter of the monochromatic components it gives.  To remedy this we give a way of associating metric spaces to hypergraphs so that distances are taken into account.
\begin{definition}
Let $V$ be a set of $n$ vertices, and $d_1, \dots, d_r$ metrics on $V$ and $\mathcal H$ an $n$-edge, $r$-uniform, $r$-partite hypergraph. 
An injection $\phi: E(\mathcal H)\to V$ is called an $(a,b)$-duality if for distinct edges $e,f$:
\begin{itemize}
\item  $e$ and $f$ intersect in part $i$ $\implies$ $d_i(\phi(e),\phi(f))\leq a$.
\item  $e$ and $f$ don't intersect in part $i$ $\implies$ $d_i(\phi(e),\phi(f))\geq b$.
\end{itemize}
\end{definition}
We say that $(V, d_1, \dots, d_r)$ \emph{contains} an $(a,b)$-dual of $\mathcal H$ if there is some $(a,b)$-duality $\phi:E(\mathcal H) \to V$.
We say that $(V, d_1, \dots, d_r)$ \emph{is} an $(a,b)$-dual
 of $\mathcal H$ if additionally $\phi$ is a bijection. 
In all the $(a,b)$-dualities  we consider, we'll always have $a<b$. It is useful to note that when $a<b$, one can replace both ``$\implies$'' in the definition of ``$(a,b)$-duality'' with ``$\iff$''.

This definition has two easy monotonicity properties which we will use. 
\begin{observation}[Monotonicity of parameters]\label{Observation_monotonicity_subgraph}
For numbers $a\leq a'<b'\leq b$, family of metrics  $(V, d_1, \dots, d_r)$, and hypergraph $\mathcal H$, if $\phi:E(\mathcal H)\to V$ is an $(a,b)$-duality, then it is also a $(a',b')$-duality.
\end{observation}
\begin{proof}
Suppose $\phi:E(H)\to V$ satisfies the definition of $(a,b)$-duality. If $e,f$ intersect in part $i$, then $d_i(\phi(e), \phi(f))\leq a\leq a'$.
If $e,f$ don't intersect in part $i$, then $d_i(\phi(e), \phi(f))\geq b\geq b'$. 
\end{proof}

\begin{observation}[Monotonicity under hypergraph containment]\label{Observation_monotonicity_containment}
For a subhypergraph $\mathcal{H}'\subseteq \mathcal{H}$, if $(V, d_1, \dots, d_r)$ contains an $(a,b)$-dual of $\mathcal H$, then it also contains an $(a,b)$-dual of $\mathcal H'$ 
\end{observation}
\begin{proof}
Let $\phi:E(\mathcal H)\to V$ be an $(a,b)$-duality. Then the restriction $\phi\big|_{E(\mathcal H')}$ is also an $(a,b)$-duality of $\mathcal H'$.
\end{proof}

The following lemma is the technical core of the paper. To understand it, compare the statement with the definition of Ryser stability. The conclusion of the lemma is very similar to that definition --- it guarantees one of two alternatives --- either an efficient cover or the containment of some object. The difference between the lemma and the definition of Ryser stability is that the lemma takes distances into account (via the terminology of $(a,b)$-duals), which is where we get control over distances in Theorem~\ref{Theorem_main}.
\begin{lemma}\label{Lemma_Stability_Transferrence}
Let $r\geq 2$ and $k> 2^{9r}$.
Suppose that $\mathcal H$ is $c$-Ryser-stable relative to $\mathcal H_1, \dots, \mathcal H_{\ell}$. 
Let $(V,d_1, \dots, d_r)$ be a family of metrics containing a $(m,km)$-dual of $\mathcal{H}$. Then one of the following holds:
\begin{itemize}
\item $V$ can be covered by $c$ radius $km$ balls chosen from the metrics  $d_1, \dots, d_r$.
\item $(V, d_1, \dots, d_r)$ contains a $(m', k^{\frac{1}{4r}}m')$-dual of one of $\mathcal H_1, \dots, \mathcal H_{\ell}$ for some $m\leq m'\leq km$.
\end{itemize}
\end{lemma}
\begin{proof}
Without loss of generality, we may suppose that $\mathcal H$ has no isolated vertices (given a family of metrics which contain a $(a,b)$-dual of hypergraph, they also contain a $(a,b)$-dual to that hypergraph minus any isolated vertices, via Observation~\ref{Observation_monotonicity_containment}).
Let $C=\{u_1, \dots, u_c\}$ be a cover witnessing the $c$-Ryser-stability of $\mathcal H$ relative to $\mathcal H_1, \dots, \mathcal H_{\ell}$.
Let $\phi:E(\mathcal{H})\to V$ be an $(m,km)$-duality (which exists because $(V, d_1, \dots, d_r)$  contains an $(m,km)$-dual of $\mathcal H$).
For each $j=1, \dots, c$ let $p_j$ be the partition containing $u_j$, and let $e_j$ be an arbitrary edge through $u_j$ (we allow repetition between the $p_j$s, and between the $e_j$s).
For each $j=1, \dots, c$ let $B_{km}(\phi(e_j))$ be the radius $km$ ball in the metric $d_{p_j}$ around the vertex $\phi(e_j)\in V$.
If $B_{km}(\phi(e_1)), \dots, B_{km}(\phi(e_c))$ cover $V$, then we are done. Thus we can suppose that there is some vertex $v\not \in B_{km}(\phi(e_1))\cup \dots \cup B_{km}(\phi(e_c))$.
Notice that this implies $v\not \in \mathrm{Im}(\phi)$ --- indeed if there was some $e$ with $\phi(e)=v$, then  since $C$ covers $e$, there would be some $u_t\in e$. Since $u_t\in e_t$ and is in part $p_t$, we get that $e$ and $e_t$ intersect in part $p_t$. Since $\phi$ is a $(m,km)$-duality of $\mathcal{H}$ we would get $d_{p_t}(\phi(e_t), v)=d_{p_t}(\phi(e_t), \phi(e))\leq m\le km$, contradicting $v\not\in B_{km}(\phi(e_t))$.

For each $t=1, \dots, r$, add a new vertex $u^*_t$ to partition $t$ of $\mathcal H$. 
For each $t=1, \dots, r$, let $m_t=\min_{f\in E(\mathcal H)} d_t(\phi(f), v)$ and let $f_t$ be an edge with $d_t(\phi(f_t), v)=m_t$. Without loss of generality, we may suppose that the partitions and metrics are ordered so that $m_1\leq m_2\leq \dots\leq m_r$. Fix $m_0=m$ and $m_{r+1}=km$. 
\begin{claim}\label{Claim_m}
There is some $t'\in \{0, \dots, r\}$ for which  $m_{t'+1}> k^{\frac1{3r}}m$, $m_{t'+1}> k^{\frac1{3r}}m_{t'}$, and $m_{t'}\leq \sqrt k m$.
\end{claim}
\begin{proof}
Let $s\in \{0,\dots, r+1\}$ be the largest index with $m_s\leq k^{\frac1{3r}}m$, noting that $s$ exists (since $m_0=m\leq k^{\frac1{3r}}m$) and that $s\leq r$ (since $m_{r+1}=km> k^{\frac1{3r}}m$).
Let $t'\in \{s,\dots, r\}$ be the smallest index for which $m_{t'+1}> k^{\frac{1}{3r}}m_{t'}$ holds.
To see that such an index must exist, note that if it didn't we'd have $m_{i+1}\leq k^{\frac{1}{3r}}m_{i}$ for $i=s, \dots, r$. Putting all these inequalities together would give $m_{r+1}\leq k^{\frac{1}{3r}}m_{r}\leq (k^{\frac{1}{3r}})^2m_{r-1}\leq (k^{\frac{1}{3r}})^3m_{r-2}\leq \dots \leq (k^{\frac{1}{3r}})^{r+1-s}m_{s} \leq 
(k^{\frac{1}{3r}})^{r+1}k^{\frac1{3r}}m= 
(k^{\frac{1}{3r}})^{r+2}m<km$, which is a contradiction. This immediately gives $m_{t'+1}> k^{\frac1{3r}}m$ (since $t'+1>s$), and $m_{t'+1}> k^{\frac1{3r}}m_{t'}$ as required by the claim.
Note that by minimality of $t'$, we have $m_{i+1}\leq k^{\frac{1}{3r}}m_{i}$ for $i=s, \dots,t'-1$. Combining all these inequalities gives $m_{t'}\leq k^{\frac{1}{3r}}m_{t'-1}\leq (k^{\frac{1}{3r}})^2m_{t'-2}\leq (k^{\frac{1}{3r}})^3m_{t'-3}\leq \dots \leq (k^{\frac{1}{3r}})^{t'-s}m_{s} \leq (k^{\frac{1}{3r}})^{r}k^{\frac1{3r}}m=k^{\frac{r+1}{3r}}m\leq \sqrt k m$, verifying the claim.
\end{proof}

Fix $m'=m_{t'}+m$ and notice that using $k>2^{9r}$, we have 
\begin{equation}\label{Eq_km_bound}
m\leq m'< k^{\frac1 {4r}} m'\leq  k^{\frac1 {4r}}(1+\sqrt k)m\leq 2k^{\frac1 {4r}+\frac 12} m 
\leq (k-\sqrt{k})m<km.
\end{equation}
Define an edge $e'$ as follows: In partitions $t=1, \dots, t'$, $e'$ has the vertex of $f_t$ in partition $t$. In partitions $t=t'+1, \dots, r$, $e'$ has the vertex $u^*_t$. 
Let $\mathcal H'=\mathcal H+e'$.
Define a function $\phi':E(\mathcal H')\to V$ by $\phi'(f)=\phi(f)$ for $f\neq e$ and $\phi'(e)=v$. This is clearly an injection since $\phi$ was an injection and $v\not \in \mathrm{Im}(\phi)$. 
\begin{claim}\label{claim_dual}
$\phi'$ is a $(m', k^{\frac{1}{4r}}m')$-duality of $\mathcal H'$. 
\end{claim}
\begin{proof}
Since $\phi'$ is an injection, we only need to check that for every pair of edges $e,f\in \mathcal{H}'$, the quantity $d_i(\phi(e), \phi(f))$ behaves as in the definition of ``$(m', k^{\frac{1}{3r}}m')$-dual''.
For a pair of edges $e, f \in \mathcal H$, the definition $\phi'$ being a  $(m', k^{\frac{1}{4r}}m')$-duality holds  by Observation~\ref{Observation_monotonicity_subgraph} as a consequence of $\phi$ being a $(m,km)$-dual and $m\leq m'< k^{\frac{1}{4r}}m'\leq km$ (which is part of (\ref{Eq_km_bound})). Therefore we can assume that one of the two edges is $e'$.
Let $f\neq e'$ be an edge of $\mathcal H'$ (or equivalently $f\in \mathcal H$). 

Suppose that $f$ and $e'$ intersect in partition $t$.  Note $u_t^*\not\in f$ (since $f$ is an edge of $\mathcal H$,  none of which contain $u_t^*$), and so we have $u_t^*\not\in e'$ also (since $f$ and $e'$ have the same vertex in partition $t$). By the definition of $e'$, this tells us that $t\leq t'$.
This implies that $e'$ intersects both $f$ and $f_t$ in partition $t$, and hence that   $f$ and $f_t$ intersect in partition $t$. By the definition of $\phi$ being a $(m, km)$-dual, this gives us   $d_t(\phi(f_t), \phi(f))\leq m$.  Recall that $d_t(\phi(f_t), v)= m_t$. By the triangle inequality  $$d_t(\phi'(f), \phi'(e'))=d_{t}(\phi(f),v)\leq d_{t}(v, \phi(f_t))+d_{t}( \phi(f_t), \phi(f))\leq m_t+m\leq  m_{t'}+m=m'.$$ 
The first equation comes from the definition of $\phi'$. The first inequality is the triangle inequality. The second inequality is $d_t(\phi(f_t), v)= m_t$ and $d_t(\phi(f_t), \phi(f))\leq m$. The third inequality is $t\leq t'$ and $m_1\leq \dots\leq m_r$.
This proves the first part of  $\phi$ being  a $(m', k^{\frac{1}{4r}}m')$-dual of $\mathcal H'$.

Suppose that $f$ and $e'$ do not intersect in partition $t$. 
Suppose that $u_t^*\in e'$ or equivalently $t\geq t'+1$. By definition of $t'$ in Claim~\ref{Claim_m}, we have $$d_t(\phi'(f), \phi'(e'))=d_t(\phi(f), v)\geq m_t\geq m_{t'+1}> \frac 12(k^{\frac{1}{3r}}m_{t'}+k^{\frac{1}{3r}}m)=k^{\frac{1}{3r}}\frac {m'}2\geq  k^{\frac{1}{4r}}m'.$$ 
\noindent 
Here the first inequality comes from the definition of $m_t$. The second inequality uses $t\geq t'+1$ and $m_1\leq \dots\leq m_r$. The third inequality is the average of two of the inequalities given by Claim~\ref{Claim_m}. The last inequality uses $k> 2^{9r}$.

Suppose that $e'$ intersects $f_t$ in partition $t$, or equivalently $t\leq t'$. Since $f$ and $e'$ don't intersect in partition $t$, $f$ and $f_t$ also don't intersect in partition $t$. Since $\phi$ is an $(m, km)$-dual of $\mathcal H$ this gives $d_t(\phi(f), \phi(f_t))\geq km$. We also have $d_t(\phi(f_t), v)= m_t$ by definition of $f_t$. By the triangle inequality we have 
\begin{align*}
d_t(\phi'(f), \phi'(e'))&=d_t(\phi(f), v)\geq d_t(\phi(f), \phi(f_t))-d_t(\phi(f_t), v)\\
&\geq km-m_t\geq km-m_{t'}\geq km-\sqrt k m\geq k^{\frac{1}{4r}}m'
\end{align*}
Here the second inequality is $d_t(\phi(f), \phi(f_t))\geq km$ together with $d_t(\phi(f_t), v)= m_t$. The third inequalilty is $t\leq t'$ and $m_1\leq \dots\leq m_r$. The fourth inequality is part of Claim~\ref{Claim_m}.
The last inequality is part of (\ref{Eq_km_bound}).  This proves the second part of  $\phi$ being  a $(m', k^{\frac{1}{4r}}m')$-dual of $\mathcal H'$.
\end{proof}
Recall that we have a cover $C$ of $\mathcal H$.
\begin{claim}
$C$ is not a cover of $\mathcal H'$
\end{claim}
\begin{proof}
Suppose for contradiction, that $C$ is a cover of $\mathcal H'$. Let $u_t$ be a vertex of $C$ contained in the edge $e'$.  Recall that $p_t$ is the partition containing $u_t$ and $e_t$ is an edge of $\mathcal H$ through $u_t$. Thus, $e_t,e'$ intersect at the vertex $u_t$ in partition $p_t$.
By Claim~\ref{claim_dual} and the definition of ``$(a,b)$-duality'', we have $d_{p_t}(\phi'(e_t), \phi'(e'))\leq m'< km$. Since $\phi'(e_t)=\phi(e_t)$ and $\phi'(e')=v$, this contradicts  $v\not\in B_{km}(\phi(e_t))$.
\end{proof}
 By the definition of $c$-Ryser-stability, $\mathcal H'$ contains a copy of the hypergraph $\mathcal H_s$ for some $s$. But then $\phi'$ restricted to this copy of $\mathcal H_s$ contains a $(m', k^{\frac1{4r}}m')$-dual of $\mathcal H_s$, so the lemma holds.
\end{proof}
The following is the first equivalence we prove between Ryser's Conjecture and coverings of bounded diameter. Using Lemma~\ref{Lemma_RyserStabilityEquivalence} we get that for any $\nu,r$ for which Ryser's Conjecture is true we have that ``let $(V, d_1, \dots, d_r)$ be a family of metrics such that every size $\nu+1$ subset of $V$ contains some distinct $u,v$ with $d_i(u,v)\leq 1$. 
Then there is a family of radius $9^{4r^{\ell+3}}$ balls $B_1, \dots, B_{(r-1)\nu}$ which covers $V$ (with each $B_j$ a ball in one of the metrics $d_i$)'' is also true. 
\begin{lemma}\label{Lemma_transferrence_metrics}
For fixed $\nu, r\in \mathbb N$, suppose that there exists a $(r-1)\nu$-Ryser-stable sequence relative to $\mathcal{M}_{r, \nu+1}$ having length $\ell$ and ending with the single-edge hypergraph $\mathcal E_r$.

Let $(V, d_1, \dots, d_r)$ be a family of metrics such that every size $\nu+1$ subset of $V$ contains some distinct $u,v$ with $d_i(u,v)\leq 1$. 
Then there is a family of radius $9^{4r^{\ell+3}}$ balls $B_1, \dots, B_{(r-1)\nu}$ which covers $V$ (with each $B_j$ a ball in one of the metrics $d_i$).
\end{lemma}
\begin{proof}
Let $\mathcal H_1, \dots, \mathcal H_{\ell}$ be a sequence which is $(r-1)\nu$-Ryser-stable  relative to  $\mathcal{M}_{r, \nu+1}$ and has $\mathcal H_{\ell}=\mathcal E_r$.
Fix $k_i=9^{(4r)^{i+1}}$ and $m_i=\prod_{j=i+1}^{\ell+1}k_i$ noting that $m_0\geq m_1\geq \dots$ and $k_i\geq 2^{9r}$ for $i\geq 0$.
First notice that $(V, d_1, \dots, d_r)$ does not contain a $(m, k_{0} m)$-dual of $\mathcal{M}_{r, \nu+1}$ for any $1\leq m\leq m_0$. Indeed suppose for contradiction that we had such a  $(m, k_{0} m)$-duality $\phi:E(\mathcal M_{r, \nu+1})\to V$. Note $|Im(\phi)|=\nu+1$ (since $\phi$ is an injection and $e(\mathcal M_{r, \nu+1})=\nu+1$). By assumption, there are distinct $u,v\in Im(\phi)$ with $d_i(u,v)\leq 1$ for some $i$. Let $e,f\in E(\mathcal M_{r, \nu+1})$ with $\phi(e)=u$, $\phi(f)=v$, noting that $e,f$ are distinct since $\phi$ is injective. Since $\mathcal M_{r, \nu+1}$ is a matching, $e,f$ don't interesect, and so the definition of ``$(m, k_{0} m)$-duality'' implies that $d_i(\phi(e), \phi(f))\geq k_0m>1$  for all $i$. This contradicts $d_i(\phi(e), \phi(f))=d_i(u,v)\leq 1$.

Since $\mathcal H_{\ell}=\mathcal E_r$ is the single-edge hypergraph,  $(V, d_1, \dots, d_r)$ contains a $(m_{\ell}, k_{\ell} m_{\ell})$-dual of $\mathcal H_{\ell}$ (any function $\phi:E(\mathcal E_r)\to V$ satisfies the definition of ``duality'').
Choose $i$ as small as possible so that $(V, d_1, \dots, d_r)$ contains a $(m, k_im)$-dual of $\mathcal H_i$ for some $1\leq m\leq m_i$.
By minimality, notice that $(V, d_1, \dots, d_r)$ does not contain a $(m', k_i^{\frac{1}{4r}} m')$-dual of $\mathcal H_j$ for any $j<i$ and $m\leq m'\leq k_im$ --- such a dual would also be a $(m', k_{i-1}m')$-dual of $H_j$ for $1\leq m'\leq m_{i-1}$ (note first that if $m\leq m'\leq k_im$ then $1\leq m\leq m'\leq k_im\leq k_im_i=m_{i-1}$. 
Also, an $(m', k_i^{\frac{1}{4r}} m')$-dual is an $(m', k_{i-1}m')$-dual by  $k_{i-1}=9^{(4r)^{i}}=(9^{(4r)^{i+1}})^{\frac1{4r}}=k_i^{\frac{1}{4r}}$). 
 By Lemma~\ref{Lemma_Stability_Transferrence}, $V$ can be covered by $(r-1)\nu$ radius $k_im_i$ balls chosen from the metrics  $d_1, \dots, d_r$ (by Observation~\ref{Observation_monotonicity_containment}), proving the lemma.
\end{proof}

We translate the above into a statement about graphs.

\begin{lemma}\label{Lemma_main_graphs}
For fixed $\nu, r\in \mathbb N$, suppose that there exists a $(r-1)\nu$-Ryser-stable sequence relative to $\mathcal{M}_{r, \nu+1}$ having length $\ell$ and ending with the single-edge hypergraph $\mathcal E_r$.
Then every $r$-edge-coloured graph $G$ with independence number $\leq \nu$ can be covered by $(r-1)\nu$ monochromatic components of diameter $\leq 2\cdot 9^{4r^{\ell+3}}$.
\end{lemma}
\begin{proof}
Let $G$ be an $r$-edge-coloured graph with independence number $\leq \nu$. By Lemma~\ref{Lemma_RyserStabilityEquivalence}, Ryser's Conjecture holds for $r$ and $\nu$. Using Proposition~\ref{Proposition_HypergraphEquivalence}, this tells us that $G$ can be covered by $(r-1)\nu$ monochromatic components. If $|V(G)|\leq 2\cdot 9^{4r^{\ell+3}}$, then these automatically have diameter $\leq 2\cdot 9^{4r^{\ell+3}}$, so we can assume that $|V(G)|> 2\cdot 9^{4r^{\ell+3}}/2$.

For each $i$, let $G_i$ be the subgraph formed by colour $i$ edges and let  $d_i:=d_{G_i}$ be the corresponding graph metric on $V:=V(G)$ as defined in (\ref{Eq_GraphMetric}). Notice that for any set $S\subseteq V$ of size $\nu+1$, there must be an edge $uv$ with $u,v\in S$ (since the independence number of $G$ is $\leq \nu$), and so $d_i(u,v)=1$ for some $i$.  Therefore Lemma~\ref{Lemma_transferrence_metrics} applies to $(V, d_1, \dots, d_r)$, producing   a family of radius $9^{4r^{\ell+3}}$ balls $B_1, \dots, B_{(r-1)\nu}$ which covers $V$. 

Let $d_j$ be the metric in which $B_j$ is a ball in. Notice that $B_j$ must be contained in a colour $j$ connected component of $G$ (since otherwise, by the definition of the graph metric in (\ref{Eq_GraphMetric}), we'd have two vertices $u,v\in B_j$ with $d_{j}(u,v)=|V(G)|\geq 2\cdot 9^{4r^{\ell+3}}$. This would contradict $B_j$ having radius $9^{4r^{\ell+3}}$).
\end{proof}
Combined with Lemma~\ref{Lemma_RyserStabilityEquivalence}, this gives our main theorem. 
\begin{proof}[Proof of Theorem~\ref{Theorem_main}]
(ii) $\implies$ (i) is immediate from how we stated Ryser's Conjecture, so we just need to prove (i) $\implies$ (ii). If (i) holds for $r$ and $\alpha$, then Lemma~\ref{Lemma_RyserStabilityEquivalence} gives us a  $(r-1)\nu$-Ryser-stable sequence $\mathcal H_1, \dots, \mathcal H_{\ell}$ relative to $\mathcal{M}_{r, \nu+1}$ such that the last hypergraph $\mathcal H_{\ell}$ in the sequence is just a single edge $\mathcal E_r$, and the sequence has length $\ell\leq 2^{(r+\alpha)^{2r}}$. Now Lemma~\ref{Lemma_main_graphs} applies (with $\nu=\alpha$) to tell us that ``every $r$-edge-coloured graph $G$ with independence number $\leq \alpha$ can be covered by $(r-1)\alpha$ monochromatic components of diameter $\leq 2\cdot 9^{4r^{\ell+3}}\le 2\cdot 9^{4r^{ 2^{(r+\alpha)^{2r}}+3}}\le  9^{r^{2^{(r+\alpha)^{4r}}}}$'' as required. 
\end{proof}

\section{Concluding remarks}
Following the results of this paper, there is limited further progress that can be made on the conjectures of Mili\'cevi\'c's and DeBiasio-Kamel-McCourt-Sheats without first making new progress on Ryser's Conjecture. Perhaps the most interesting open problem (which isn't reliant on a breakthrough in Ryser's conjecture), would be to improve the bounds in some of the Corollaries~\ref{Corollary1} --~\ref{Corollary3}. Various small improvements are easy to do e.g. by using better bounds on the Sunflower Lemma from~\cite{alweiss2020improved}. Another way to improve the bound, pointed out by Zach Hunter, is to note that the ``length of a Ryser stable sequence'' isn't really the natural parameter to use in the proof of Lemma~\ref{Lemma_transferrence_metrics} (and Lemma~\ref{Lemma_Stability_Transferrence}). What the proof really uses would be more naturally called the ``depth'' of the Ryser stable sequence. To define it first construct an acyclic digraph on $V(D)=\{H_1, \dots, H_t\}$ (where $H_1, \dots, H_t$ is the Ryser stable) with each $H_i$ joined to some subset $S\subseteq \{H_1, \dots, H_{i-1}\}$ such  that $H_i$ is Ryser-stable relative to $S$. The depth of the sequence is defined as the length of the longest path in $D$. Since it can be shown that the Ryser stable sequence in Lemma~\ref{Lemma_RyserStabilityEquivalence} has depth $r!((\nu + 1)r)r+1$, this allows one to drop a single exponential from all the bounds in the paper.

For $r=2$, DeBiasio, Girão,  Haxell,  and Stein~\cite{debiasio2025bounded} showed that one can cover every $2$-edge-coloured  $G$ by $\alpha(G)$ monochromatic components whose diameters $\le 8\alpha(G)^2+12\alpha(G)+6$, which  suggests that for general $r$, the bound should be polynomial also.
The most fascinating open problem here is to understand whether \emph{any dependence on $\alpha$ or $r$ is necessary at all}. Perhaps there is some absolute constant $d$ such that every $r$-edge-coloured graph $G$ can be covered by $(r-1)\alpha(G)$ monochromatic components of diamater $\le d$?

 \subsection*{Acknowledgement}
 The author would like to thank Ahmad Ahu-Khazneh for many discussions related to this project. The author would also like to thank Louis DeBiasio and Zach Hunter for carefully reading the paper and suggesting many improvements and corrections.

\bibliographystyle{abbrv}
\bibliography{ryser_bib}

\begin{thebibliography}{10}

\bibitem{abu2016matchings}
A.~Abu-Khazneh.
\newblock {\em Matchings and covers of multipartite hypergraphs}.
\newblock PhD thesis, London School of Economics and Political Science, 2016.

\bibitem{ahro}
R.~Aharoni.
\newblock Ryser's conjecture for tripartite 3-graphs.
\newblock {\em Combinatorica}, 21(1):1--4, 2001.

\bibitem{alweiss2020improved}
R.~Alweiss, S.~Lovett, K.~Wu, and J.~Zhang.
\newblock Improved bounds for the sunflower lemma.
\newblock In {\em Proceedings of the 52nd Annual ACM SIGACT Symposium on Theory
  of Computing}, pages 624--630, 2020.

\bibitem{austin2005contractive}
T.~D. Austin.
\newblock On contractive families and a fixed-point question of {S}tein.
\newblock {\em Mathematika}, 52(1-2):115--129, 2005.

\bibitem{best2018did}
D.~Best and I.~M. Wanless.
\newblock What did ryser conjecture?
\newblock {\em arXiv preprint arXiv:1801.02893}, 2018.

\bibitem{debiasio2025bounded}
L.~DeBiasio, A.~Gir{\~a}o, P.~Haxell, and M.~Stein.
\newblock A bounded diameter strengthening of {K}{\H{o}}nig’s theorem.
\newblock {\em SIAM Journal on Discrete Mathematics}, 39(2):1269--1273, 2025.

\bibitem{debiasio2020generalizations}
L.~DeBiasio, Y.~Kamel, G.~McCourt, and H.~Sheats.
\newblock Generalizations and strengthenings of ryser's conjecture.
\newblock {\em arXiv preprint arXiv:2009.07239}, 2020.

\bibitem{duchet1979representations}
P.~Duchet.
\newblock {\em Repr{\'e}sentations, noyaux en th{\'e}orie des graphes et
  hypergraphes}.
\newblock PhD thesis, {\'E}diteur inconnu, 1979.

\bibitem{english2021low}
S.~English, C.~Mattes, G.~McCourt, and M.~Phillips.
\newblock Low diameter monochromatic covers of complete multipartite graphs.
\newblock {\em arXiv preprint arXiv:2105.07038}, 2021.

\bibitem{erdHos1991vertex}
P.~Erd{\H{o}}s, A.~Gy{\'a}rf{\'a}s, and L.~Pyber.
\newblock Vertex coverings by monochromatic cycles and trees.
\newblock {\em Journal of Combinatorial Theory, Series B}, 51(1):90--95, 1991.

\bibitem{gyarfas1977partition}
A.~Gy{\'a}rf{\'a}s.
\newblock Partition coverings and blocking sets in hypergraphs.
\newblock {\em Communications of the Computer and Automation Institute of the
  Hungarian Academy of Sciences}, 71:62, 1977.

\bibitem{gyarfas20252}
A.~Gyarfas and G.~N. Sarkozy.
\newblock 2-reachable subsets in two-colored graphs.
\newblock {\em arXiv preprint arXiv:2506.11696}, 2025.

\bibitem{gyarfas2025bounded}
A.~Gyarfas and G.~N. Sarkozy.
\newblock Bounded diameter variations of ryser's conjecture.
\newblock {\em arXiv preprint arXiv:2505.02564}, 2025.

\bibitem{hend}
J.~R. Henderson.
\newblock {\em Permutation Decomposition of (0,1)-Matrices and Decomposition
  Transversals}.
\newblock PhD thesis, Caltech, 1971.

\bibitem{milicevic2015commuting}
L.~Mili{\'c}evi{\'c}.
\newblock Commuting contractive families.
\newblock {\em Fundamenta Mathematicae}, 231(3):225--272, 2015.

\bibitem{milicevic2019covering}
L.~Mili{\'c}evi{\'c}.
\newblock Covering complete graphs by monochromatically bounded sets.
\newblock {\em Applicable Analysis and Discrete Mathematics}, 13(1):85--110,
  2019.

\bibitem{tuz_un}
Z.~Tuza.
\newblock Some special cases of {R}yser's conjecture.
\newblock {\em unpublished manuscripts}, 1979.

\bibitem{tuz2}
Z.~Tuza.
\newblock Ryser's conjecture on transversal of $r$-partite hypergraphs.
\newblock {\em Ars Combinatoria}, 16:201--209, 1983.

\end{thebibliography}

\end{document}